\documentclass[11pt]{article}
\usepackage{latexsym,amsmath,amssymb,amsfonts,epsfig,graphicx,cite,psfrag}
\usepackage{eepic,color,colordvi,amscd}
\usepackage{mathrsfs}
\usepackage{amsthm}
\usepackage{enumerate}
\usepackage{indentfirst}
\usepackage{tikz}
\usepackage{subfigure}
\usepackage[english]{babel}
\usepackage[center]{caption2}
\usepackage{multirow}
\usepackage[usenames,dvipsnames]{pstricks}
\usepackage{pst-grad} 
\usepackage{pst-plot} 
\usepackage[space]{grffile} 
\usepackage{etoolbox} 
\usepackage{float}
\usepackage{subfigure}
\makeatletter 
\patchcmd\Gread@eps{\@inputcheck#1 }{\@inputcheck"#1"\relax}{}{}

\renewcommand{\thesubfigure}{(\roman{subfigure})}
\makeatletter \renewcommand{\@thesubfigure}{\thesubfigure \space}

\newtheorem{theo}{Theorem}[section]

\newtheorem{lem}[theo]{Lemma}

\newtheorem{prop}[theo]{Proposition}
\newtheorem{conj}[theo]{Conjecture}
\newtheorem{defi}[theo]{Definition}

\newtheorem{prob}[theo]{Problem}

\begin{document}

\title{\bf\Large Phase Transitions of Structured Codes of Graphs}

\date{July 11, 2023}

\author{
Bo Bai\thanks{Theory Lab, Central Research Institute, 2012 Labs, Huawei Technologies Co., Ltd. China. Email: baibo8@huawei.com.}~~~~~~~
Yu Gao\thanks{Theory Lab, Central Research Institute, 2012 Labs, Huawei Technologies Co., Ltd. China. Email: gaoyu99@huawei.com.}~~~~~~~
Jie Ma\thanks{School of Mathematical Sciences, University of Science and Technology of China, Hefei, Anhui, 230026, China. Resserach supported by National Key Research and Development Program of China 2020YFA0713100, National Natural Science Foundation of China grant 12125106, and Anhui Initiative in Quantum Information Technologies grant AHY150200. Email: jiema@ustc.edu.cn.}
\thanks{Hefei National Laboratory, University of Science and Technology of China, Hefei 230088, China. Research supported by Innovation Program for Quantum Science and Technology 2021ZD0302902.}~~~~~~~
Yuze Wu\thanks{School of Mathematical Sciences, University of Science and Technology of China, Hefei, Anhui, 230026, China. Email: lttch@mail.ustc.edu.cn.}}

\maketitle

\begin{abstract}

	We consider the symmetric difference of two graphs on the same vertex set $[n]$, which is the graph on $[n]$ whose edge set consists of all edges that belong to exactly one of the two graphs. Let $\mathcal{F}$ be a class of graphs, and let $M_{\mathcal{F}}(n)$ denote the maximum possible cardinality of a family $\mathcal{G}$ of graphs on $[n]$ such that the symmetric difference of any two members in $\mathcal{G}$ belongs to $\mathcal{F}$. These concepts are recently investigated by Alon, Gujgiczer, K\"{o}rner, Milojevi\'{c}, and Simonyi, with the aim of providing a new graphic approach to coding theory. In particular, $M_{\mathcal{F}}(n)$ denotes the maximum possible size of this code.

Existing results show that as the graph class $\mathcal{F}$ changes, $M_{\mathcal{F}}(n)$ can vary from $n$ to $2^{(1+o(1))\binom{n}{2}}$. We study several phase transition problems related to $M_{\mathcal{F}}(n)$ in general settings and present a partial solution to a recent problem posed by Alon et. al.
\end{abstract}

\section{Introduction}

\noindent For decades the application of graph theory in coding theory has been a significant and productive area of research.
Gallager's early work\cite{Gal} in 1963 illustrated the potential of using graphs to construct codes with desirable properties.
This technique was significantly advanced by Tanner \cite{Tan} and further developed by Sipser and Spielman\cite{SS} (for constructing expander codes -- the first explicit code of this kind).
Since then, the application of graph-based techniques has brought about many significant discoveries in coding theory (e.g.\cite{AS},\cite{BE},\cite{BV},\cite{BZ},\cite{EK},\cite{KKL},\cite{LR},\cite{LZ},\cite{RU},\cite{Zem}),
including the recent explicit construction of locally testable codes with constant rate, constant distance, and constant locality\cite{DELLM},\cite{PK}.
	
A recent study by  Alon, Gujgiczer, K\"{o}rner, Milojevi\'{c} and Simonyi\cite{AKGMS} explores a graph-theoretic variation of the following basic problem on code distance:
how many binary sequences of a given length can be found if any two of them differ in at least a given number of coordinates?
Instead of prescribing the minimum distance of two codewords, the authors require the codewords differ in some specific structure.
In particular, they are interested in the largest family of graphs on the same vertex set such that the symmetric difference of the edge sets of any two graphs in this family has certain desired property.
	
To be formal, let $\mathcal{F}$ be a fixed class of graphs. A family $\mathcal{G}$ of graphs on the same vertex set $[n]$, where $[n]=\{1,2\dots,n\}$, is called {\it $\mathcal{F}$-good} if for any pair $G,G'\in\mathcal{G}$, their symmetric difference $G\oplus G'$, the graph with the vertex set $[n]$ and the edge set
$$E(G\oplus G')=(E(G)\setminus E(G'))\cup(E(G')\setminus E(G))$$ belongs to $\mathcal{F}$.
This family $\mathcal{G}$ is also called an {\it $\mathcal{F}$-code} since it can be viewed as a $\{0,1\}$-code of length $\binom{n}{2}$.
Let $M_{\mathcal{F}}(n)$ denote the maximum possible size of an $\mathcal{F}$-good family on $[n]$.
When the graph class $\mathcal{F}$ consists of all the graphs containing a fixed graph $L$, we use {\it $L$-code} and $M_{L}(n)$ instead of $\mathcal{F}$-code and $M_{\mathcal{F}}(n)$.
It is worth noting that there is a hidden relationship between $\mathcal{F}$-codes and expander codes that are extensively utilized.
In an expander code, each codeword corresponds to a $\{0,1\}$-edge-coloring of an expander graph (i.e., a highly connected sparse graph) such that the codeword induced by the edges adjacent to the same vertex belongs to a specified linear code. Therefore if every graph in class $\mathcal{F}$ is a spanning subgraph of an $n$-vertex expander graph, then an $\mathcal{F}$-code can be considered as a variation of expander codes.
	
The authors of \cite{AKGMS} provided accurate estimations on $M_{\mathcal{F}}(n)$ for various natural families $\mathcal{F}$ of graphs.
One of their main results determines the asymptotic behaviors of the rate of maximum $L$-codes for any fixed graph $L$ with at least one edge.
Let $\chi(L)$ be the chromatic number of $L$. The authors proved that
\begin{equation}\label{equ:ML}\tag{1}
\lim_{n\to\infty}\frac{\log M_{L}(n)}{\binom{n}{2}}=\frac{1}{\chi(L)-1},
\end{equation}
where logarithms here and in the rest of this paper are base 2.
They also studied the basic problem when $\mathcal{F}$ consists of graphs containing a given spanning tree.
Let $\mathcal{F}_{H_{P}}$ denote the class of graphs containing a Hamiltonian path and $\mathcal{F}_{S}$ denote the class of graphs containing a spanning star.
The authors of \cite{AKGMS} showed that $M_{\mathcal{F}_{H_{P}}}(n)=2^{n-1}$ for infinitely many $n$ and $M_{\mathcal{F}_{S}}(n)\in\{n,n+1\}$ for all positive integers $n$.
In particular, the above results reveal that as the graph family $\mathcal{F}$ changes, the value of $M_{\mathcal{F}}(n)$ can have a significant discrepancy, ranging from $n$ to $2^{(1+o(1))\binom{n}{2}}$.\footnote{If $L$ is a bipartite graph, then \eqref{equ:ML} gives the latter bound $2^{(1+o(1))\binom{n}{2}}$.}
Inspired by this intriguing phenomenon,
in this paper we aim to investigate phase transition problems on $M_{\mathcal{F}}(n)$ in general settings.



Observe the aforementioned results on tress are about two extreme cases of spanning trees. 
The authors of \cite{AKGMS} raised the following problem.
	
\begin{prob}[Problem 3 in \cite{AKGMS}]\label{P3}
For what ``natural'' sequences $\{T_{i}\}_{i\geq 1}$ of trees (with $T_{i}$ having exactly $i$ vertices for every $i$) will the value of $M_{T_{n}}(n)$ grow only linearly in $n$?
A similar question is valid if $T_{i}$ is replaced by $\mathcal{T}_{i}$, some ``natural'' family of $i$-vertex trees.
\end{prob}

Our first result provides a partial solution to this problem, by indicating that the above problem is closely related to the number of leaves of the trees.
Let $\mathcal{F}_{\ell}$ denote the family of graphs containing a spanning tree that has exactly $\ell$ leaves.
	
\begin{theo}\label{Tree.Th}
For infinitely many $n$ and all integers $3\leq\ell \leq\frac{n-1}{12 \log n}+2$, we have
	$$M_{\mathcal{F}_{\ell}}(n)\geq 2^{n-2}$$
In particular, this holds whenever $n\geq 64$ and $n=p$ or $n=2p-1$ for odd primes $p$.
\end{theo}

Let $\mathcal{F}_{c}$ denote the class of all connected graphs. It was proved in \cite{AKGMS} that $M_{\mathcal{F}_{c}}(n)=2^{n-1}$,
which implies that $M_{\mathcal{F}_{\ell}}(n)\leq 2^{n-1}$ holds for all $2\leq\ell\leq n-1$.
Hence, Theorem~\ref{Tree.Th} is tight up to a factor of two.
It is interesting to determine the precise values of $M_{\mathcal{F}_{\ell}}(n)$.
We will discuss in the concluding remark that the proof of Theorem~\ref{Tree.Th} in fact shows that the family $\mathcal{T}_{\ell}$ consisting of all spanning trees with $\ell$ leaves for any $2\leq \ell\leq \frac{n+9}{6}$ can not provide a positive answer to Problem~\ref{P3}.



We now consider some ``robust'' generalizations of \eqref{equ:ML}.
For a given graph $L$, let $M_{L}(n,k)$ denote the largest cardinality of a family $\mathcal{G}$ of graphs on $[n]$, such that the symmetric difference of any two members of $\mathcal{G}$ contains at least $k$ copies of $L$.
Let $v(L)$ and $e(L)$ denote the number of vertices and edges in $L$, respectively.
The following result shows that there is a phase transition in terms of $k=k(n)$ for the asymptotic behavior of $\lim_{n\to\infty}\log M_{L}(n,k)/\binom{n}{2}$.
	
\begin{theo}\label{kcopy.Th}
	Let $L$ be any graph with at least one edge. If $k=o(n^{v(L)})$, then we have
	$$\lim_{n\to\infty}\frac{\log M_{L}(n,k)}{\binom{n}{2}}=\frac{1}{\chi(L)-1}.$$
	If $k=cn^{v(L)}$ for some constant $c>0$, then we have
	$$\lim_{n\to\infty}\frac{\log M_{L}(n,k)}{\binom{n}{2}}\leq \frac{1}{\chi(L)-1}-\frac{2c}{e(L)}.$$	
\end{theo}

We also consider the corresponding version for vertex-disjoint copies and obtain the following phase transition.
For a given graph $L$, let $M_{k\cdot L}(n)$ denote the largest cardinality of a family $\mathcal{G}$ of graphs on $[n]$, such that the symmetric difference of any two members of $\mathcal{G}$ contains at least $k$ vertex-disjoint copies of $L$.

\begin{theo}\label{kvexd.Th}
	Let $L$ be any graph with at least one edge. If $k=o(n)$, then we have
	$$\lim_{n\to\infty}\frac{\log M_{k\cdot L}(n)}{\binom{n}{2}}=\frac{1}{\chi(L)-1}.$$
	If $k=cn$ for some constant $c>0$, then we have
	$$\lim_{n\to\infty}\frac{\log M_{k\cdot L}(n)}{\binom{n}{2}}\leq \frac{(1-c)^{2}}{\chi(L)-1}.$$
\end{theo}
	
For bipartite graphs, we can obtain the following analogs, which also provide some improvements over \eqref{equ:ML}.

\begin{theo}\label{ComBi.Th}
	If $t=o(\log n)$, then we have
	$$\lim_{n\to\infty}\frac{\log M_{K_{t,t}}(n)}{\binom{n}{2}}=1.$$
	If $t=c\log n$ for some constant $c>0$, then we have
	$$\lim_{n\to\infty}\frac{\log M_{K_{t,t}}(n)}{\binom{n}{2}}\leq1-2^{-\frac{2}{c}}.$$
\end{theo}	
	
\begin{theo}\label{AnoBi.Th}
	Let $L(r,m)=(A\cup B,E)$ be a connected bipartite graph on $m$ vertices such that any vertex in $A$ has at most $r$ neighbors in $B$.
	If $m=O(n^{1-\varepsilon})$ for some constant $\varepsilon>0$, then for any constant integer $r$, we have
	$$\lim_{n\to\infty}\frac{\log M_{L(r,m)}(n)}{\binom{n}{2}}=1.$$
\end{theo}	

Since any graph on vertex set $[n]$ can be viewed as a spanning subgraph of $K_{n}$, what if we replace $K_{n}$ with some alternative graphs on $[n]$? This raises the question of determining the maximum number of spanning subgraphs of a fixed graph $G$, with the restriction that the symmetric difference of any two of them belongs to a fixed graph class $\mathcal{F}$.
The most natural instance of this problem that comes to mind is
when $\mathcal{F}$ denotes the family of all connected graphs and $G$ denotes an $m\times n$ grid.
Here, an {\it $m\times n$ grid}, 
denoted by $G_{m,n}$, is the graph with vertex set $[m]\times [n]$ and with edges between $(u,v)$ and $(i,j)$ if and only if $u=i$ and $v\equiv j\pm1(\textrm{mod } n)$ or $v=j$ and $u\equiv i\pm1(\textrm{mod } m)$.
We have the following.
	
\begin{prop}\label{grid.Pro}
	For any integers $m,n\geq 3$, let $M_{\mathcal{F}_{c}}(G_{m,n})$ denote the maximum possible size of a family $\mathcal{G}$ of spanning subgraphs of $G_{m,n}$ such that the symmetric difference of any two members in $\mathcal{G}$ is connected, then we have $M_{\mathcal{F}_{c}}(G_{m,n})\leq 16$. Especially, we also have $M_{\mathcal{F}_{c}}(G_{m,n})=16$ for $m=n=3$.
\end{prop}	

The rest of the paper is organized as follows.
In Section 2, we present necessary preliminaries, including some definitions and known results.
In Section 3, we give the full proofs of our results.
Finally, in Section 4, we discuss some concluding remarks.

\section{Preliminaries}
\noindent We first state the following important definition from \cite{AKGMS}.

\begin{defi}\label{Fbad.D}
Let $\mathcal{F}$ be a family of graphs. Let $D_{\mathcal{F}}$ denote the maximum possible size of a family $\mathcal{G}$ of graphs on $[n]$ such that the symmetric difference of no two members of $\mathcal{G}$ belongs to $\mathcal{F}$.
\end{defi}

As pointed out in \cite{AKGMS}, determining $D_{\mathcal{F}}(n)$ can be referred to as the dual problem of determining $M_{\mathcal{F}}(n)$ (see \cite{Alon} for more results about $D_{\mathcal{F}}(n)$).
This is because that if we denote by $\overline{\mathcal{F}}$ the class containing exactly those graphs that do not belong to $\mathcal{F}$,
then we have $$D_{\mathcal{F}}(n)=M_{\overline{\mathcal{F}}}(n).$$
Moreover, the authors of \cite{AKGMS} established the following useful relation between $M_{\mathcal{F}}(n)$ and $D_{\mathcal{F}}(n)$.
	
\begin{lem}[Alon et. al, \cite{AKGMS}]\label{Dual.L}
	For any graph family $\mathcal{F}$ and any positive integer $n$, we have
	$$M_{\mathcal{F}}(n)D_{\mathcal{F}}(n)\leq 2^{\binom{n}{2}}.$$
\end{lem}
	
Fix a graph $L$. A graph $G$ is called {\it $L$-free} if $G$ does not contain $L$ as a subgraph. Let the Tur\'an number of $L$, denoted by $ex(n,L)$, be the maximum number of edges in an $n$-vertex $L$-free graph.
We will need the following classical results on Tur\'an numbers.
	
\begin{theo}[Erd\"{o}s-Stone\cite{ES}]\label{ES.Th}
	For any graph $L$ with at least one edge, we have
	$$ex(n,L)=(1-\frac{1}{\chi(L)-1}+o(1))\binom{n}{2}.$$
\end{theo}	

\begin{theo}[K\"{o}vari-S\'{o}s-Tur\'{a}n\cite{KST}]\label{KST.Th}
	For any integer $t\geq s\geq 2$, we have
	$$ex(n,K_{s,t})\leq\frac{1}{2}(t-1)^{1/s}n^{2-1/s}+\frac{1}{2}(s-1)n.$$
\end{theo}	

Let $F_{n}(L)$ denote the number of $L$-free graphs on $[n]$.
This is asymptotically determined by the chromatic number $\chi(L)$ of $L$ in the following theorem.

\begin{theo}[Erd\"{o}s-Frankl-R\"{o}dl\cite{EFR}]\label{EFR.Th}
	Suppose that $\chi(L)\geq 3$. Then
	$$F_{n}(L)=2^{(1-\frac{1}{\chi(L)-1}+o(1))\binom{n}{2}}.$$
\end{theo}

Lastly, we need two powerful tools from Extremal Graph Theory, known as Graph Removal Lemma and Dependent Random Choice.

\begin{lem}[Graph Removal Lemma\cite{EFR}]\label{GR.L}
	Given any fixed graph $L$, for any $\varepsilon>0$, there exists $\delta>0$ such that for any $n$-vertex graph $G$ which contains at most $\delta n^{v(L)}$ copies of $L$, we can remove at most $\varepsilon n^{2}$ edges of $G$ to get an $L$-free graph.
\end{lem}

\begin{lem}[Dependent Random Choice\cite{FS}]\label{DRC.L}
	Let $\alpha\in(0,1)$, $t,r,m,u,n$ be integers such that $\alpha^{t}n-\binom{n}{r}(\frac{m}{n})^{t}\geq u$. Then for any $n$-vertex graph $G$ with at least $\frac{\alpha}{2}n^{2}$ edges, there exists $U\subseteq V(G)$ with $|U|\geq u$ such that any $r$-set $S\subseteq U$ has at least $m$ common neighbors in $G$.
\end{lem}
	
\section{Proofs}

\subsection{Proof of Theorem \ref{Tree.Th}}

\noindent The proof of Theorem \ref{Tree.Th} is closely related to the following famous conjecture of Kotzig \cite{Kot}.
By a {\it perfect 1-factorization}, we mean the partition of the edge set of a graph into perfect matchings such that the union of any two of them forms a Hamiltonian cycle.
	
\begin{conj}[Perfect 1-factorization Conjecture, Kotzig \cite{Kot}]\label{P1FC}
	The complete graph $K_{n}$ has a perfect 1-factorization for any even $n>2$.
\end{conj}

	This conjecture is still open in general, but it is known to hold in several special cases. For example, whenever $n=p+1$(Kotzig\cite{Kot}) or $n=2p$ for some odd prime $p$ (Anderson\cite{An} and Nakamura\cite{Naka}, cf. also Kobayashi\cite{Koba}).
	
\begin{proof}[Proof of Theorem \ref{Tree.Th}]
	Let $n\geq 65$ be an odd integer such that Conjecture \ref{P1FC} holds for $n+1$, then we can partition the edge set of $K_{n+1}$ into $n$ perfect matchings $M_{1},M_{2},\dots,M_{n}$ such that the union of any two of them forms a Hamiltonian cycle in $K_{n+1}$.
	For each $1\leq i\leq n$, we delete the edge adjacent to $n+1$ in $M_{i}$, then for any $i\neq j\in[n]$, $M_{i}\cup M_{j}$ forms a Hamiltonian path in $K_{n}$.
	
	Let $\mathcal{G}$ be the graph family consists of the unions of even number of matchings in $\mathcal{M}=\{M_{1},M_{2},\dots,M_{n-1}\}$, then $|\mathcal{G}|=2^{n-2}$. Note that $M_{n}$ is moved out from $\mathcal{M}$ because we will then add some edges in $M_{n}$ to the members of $\mathcal{G}$ in order to guarantee that the symmetric difference of any two members of $\mathcal{G}$ is the disjoint union of at least 2 matching in $\mathcal{M}$ and at least $3\ell-5$ additional edges.
	
	Firstly, we want to partition $\mathcal{G}$ into several parts, with the property that the symmetric difference of any two graphs in the same part is the union of at least 4 matchings in $\mathcal{M}$.
	Let $G$ be the union of $2k$ matchings in $\mathcal{M}$, then the number of different graphs $G'$ in $\mathcal{G}$, such that the symmetric difference of $G$ and $G'$ is exactly the union of two matchings in $\mathcal{M}$, is $\binom{2k}{2}+2k(n-1-2k)+\binom{n-1-2k}{2}=\binom{n-1}{2}$. Let $F$ denote the graph whose vertices are the members of $\mathcal{G}$ and two members are connected in $F$ if and only if their symmetric difference is exactly the union of 2 matchings in $\mathcal{M}$. Then $F$ is an $\binom{n-1}{2}$-regular graph and our desired partition of $\mathcal{G}$ equals to a proper vertex coloring of $F$. Since $\chi(F)\leq \Delta(F)+1=\binom{n-1}{2}+1$, we can partition $\mathcal{G}$ into $s=\binom{n-1}{2}+1$ parts $\mathcal{G}_{1},\dots,\mathcal{G}_{s}$ such that the symmetric difference of any two graphs in the same part is the union of at least 4 matchings in $\mathcal{M}$.
	
	Secondly, we want to find a family $\mathcal{H}$ of subgraphs of $M_{n}$ such that the symmetric difference of any two members in $\mathcal{H}$ contains at least $3\ell-5$ different edges. Since $M_{n}$ consists of $\frac{n-1}{2}$ disjoint edges, we are going to find a Hamming Code $\mathcal{H}$ with length $\frac{n-1}{2}$ and minimum distance $3\ell-5$.
	By the famous Gilbert–Varshamov bound (due to Edgar Gilbert\cite{Gil} and independently Rom Varshamov\cite{Var}), there exists an $\mathcal{H}$ with cardinality at least
	$$\frac{2^{\frac{n-1}{2}}}{\sum_{i=0}^{3\ell-6}\binom{\frac{n-1}{2}}{i}}.$$
	Therefore, we have
	$$|\mathcal{H}|\geq \frac{2^{\frac{n-1}{2}}}{(3\ell-6)(\frac{n-1}{2})^{3\ell-6}}.$$
	Let $m=\frac{n-1}{2}$ and $t=3\ell-6$. To get $|\mathcal{H}|\geq \binom{n-1}{2}+1$, we only need to prove that
	$$\frac{2^{m}}{tm^{t}}\geq 2m^{2}.$$ That is
	$$m-t\log m-\log t-2\log m-1\geq 0.$$
	Since $\ell \leq\frac{n-1}{12(\log n)}+2$, we have $t\leq\frac{n-1}{4(\log n)}\leq \frac{m}{2\log m}$.
	Therefore,
	$$m-t\log m-\log t-2\log m-1\geq \frac{m}{2}-3\log m+\log\log m\geq 0$$ due to $m=\frac{n-1}{2}\geq 32$.
	
	So we can take $s=\binom{n-1}{2}+1$ different subgraphs $H_{1},H_{2},\dots,H_{s}$ of $M_{n}$ such that the symmetric difference of any two of them contains at least $3\ell-5$ disjoint edges. Let $\mathcal{G}_{i}'=\{G\cup H_{i}|G\in\mathcal{G}_{i}\}$ for all $1\leq i\leq \binom{n-1}{2}+1$ and $\mathcal{G}'=\bigcup_{i=1}^{\binom{n-1}{2}+1}\mathcal{G}_{i}'$. Then the symmetric difference of any two members of $\mathcal{G}'$ contains at least 2 matchaings in $\mathcal{M}$ and $3\ell-5$ disjoint additional edges.
	
	Finally, we only need to find a spanning tree with exactly $\ell$ leaves in the union graph $G$, which consists of 2 matchings in $\mathcal{M}$ and $3\ell-5$ disjoint additional edges. Let $T=v_{1}v_{2}\dots v_{n}$ be the Hamiltonian path consists of 2 matchings in $\mathcal{M}$ and let $E_{A}$ be the set of $3\ell-5$ additional edges. We first remove the edge adjacent to $v_{n}$ from $E_{A}$ if there exists such an edge in $E_{A}$ and then do the following operation: 
\begin{itemize}
	\item Take an edge $\{v_{i},v_{j}\}$ in $E_{A}$, where $i<j$ and $i$ is as small as possible, add this edge to $T$ and remove it from $E_{A}$. Delete the edge $\{v_{i},v_{i+1}\}$ from the $T$ and remove any edges that are adjacent to $v_{i+1}$ or $v_{j-1}$ from $E_{A}$.
\end{itemize}	
	 Note that after this operation, the number of leaves in the spanning tree $T$ increases exactly one, and we remove at most 3 edges from $E_{A}$. So we can repeat this operation $\ell-2$ times and then $T$ becomes a spanning tree with exactly $\ell$ leaves. 	
\end{proof}

\subsection{Proof of Theorem \ref{kcopy.Th}}
	
\begin{proof}[Proof of Theorem \ref{kcopy.Th}]
	First we prove the case $k=o(n^{v(L)})$.
	Let $F_{n}(L,k)$ denote the number of graphs containing at most $k-1$ copies of $L$ on $[n]$. Let $G$ be such a graph. By Lemma \ref{GR.L}, we can delete at most $o(n^{2})$ edges of $G$ to get an $L$-free graph. So we have
	$$F_{n}(L,k)\leq F_{n}(L)\binom{\binom{n}{2}}{o(n^{2})}=2^{o(1)\binom{n}{2}}F_{n}(L).$$
	On the other hand, we also have $F_{n}(L,k)\geq F_{n}(L)$.
	Then by Theorem \ref{EFR.Th} and Theorem \ref{ES.Th} we have
	$$F_{n}(L,k)=2^{(1-\frac{1}{\chi(L)-1}+o(1))\binom{n}{2}}.$$
	
	Let $G_{L,k}$ denote the graph whose vertices are all possible graphs on $[n]$ and two are connected if and only if their symmetric difference contains at most $k-1$ copies of $L$. Then $M_{L}(n,k)$ equals to the independence number $\alpha(G_{L,k})$ of $G_{L,k}$. Moreover, $G_{L.k}$ is an $F_{L}(n,k)$-regular graph. So we have
	$$M_{L}(n,k)=\alpha(G_{L,k})\geq\frac{v(G_{L,k})}{\Delta(G_{L,k})+1}=\frac{2^{\binom{n}{2}}}{F_{n}(L,k)+1}=2^{(\frac{1}{\chi(L)-1}+o(1))\binom{n}{2}}.$$
	Since $M_{L}(n,k)\leq M_{L}(n)=2^{(\frac{1}{\chi(L)-1}+o(1))\binom{n}{2}}$, we have
	$$\lim_{n\to\infty}\frac{\log M_{L}(n,k)}{\binom{n}{2}}=\frac{1}{\chi(L)-1}$$ for $k=o(n^{v(L)})$.
	
	Then we prove the case when $k=cn^{v(L)}$ for some constant $c>0$.
	Let $G$ be an arbitrary graph. When we add an edge to $G$, the number of copies of $L$ in $G$ increases at most $e(L)n^{v(L)-2}$. So we can construct a graph $H$ which contains at most $k-1$ copies of $L$ by adding $\frac{2c}{e(L)}\binom{n}{2}$ edges to an extremal $L$-free graph (i.e., an $L$-free graph with maximum number of edges). Then $e(H)=ex(n,L)+\frac{2c}{e(L)}\binom{n}{2}$ and we can obtain a lower bound of the corresponding dual concept $D_{L}(n,k)$ by constructing a family consisting of all subgraphs of $H$. Therefore, by Lemma \ref{Dual.L}, we have
	$$M_{L}(n,k)\leq2^{\binom{n}{2}-e(H)}.$$
	Then by Theorem \ref{ES.Th}, we have that for $k=cn^{v(L)}$,
	$$\lim_{n\to\infty}\frac{\log M_{L}(n,k)}{\binom{n}{2}}\leq \frac{1}{\chi(L)-1}-\frac{2c}{e(L)}$$
as desired.
\end{proof}

\subsection{Proof of Theorem \ref{kvexd.Th}}
	
\begin{proof}[Proof of Theorem \ref{kvexd.Th}]
	We first prove the case $k=o(n)$.
	Let $F_{n}(k\cdot L)$ denote the number of graphs containing $k$ vertex-disjoint copies of $L$ on $[n]$.
	Let $G$ be an arbitrary graph, a copy of $L$ in $G$ intersects at most $v(L)n^{v(L)-1}$ other copies of $L$. So we have $F_{n}(k\cdot L)\leq F_{L}(n,k(v(L)n^{v(L)-1}+1))$. Since $k=o(n)$, $k(v(L)n^{v(L)-1}+1)=o(n^{v(L)})$.
	Using $F_{n}(L)\leq F_{n}(k\cdot L)\leq F_{L}(n,k(v(L)n^{v(L)-1}+1))$, we have
	$$F_{n}(k\cdot L)=2^{(1-\frac{1}{\chi(L)-1}+o(1))\binom{n}{2}}.$$
	Then by the same argument in the proof of Theorem \ref{kcopy.Th}, we have that for $k=o(n)$
	$$\lim_{n\to\infty}\frac{\log M_{k\cdot L}(n)}{\binom{n}{2}}=\frac{1}{\chi(L)-1}.$$

Next we consider the case when $k=cn$ for some constant $c>0$.
Let $G$ be a graph on $n$ labeled vertices which satisfies the following two properties:
\begin{itemize}	
\item [(i)] If $S$ denotes the first $k-1$ vertices of $G$, then $G[S]$ is a clique and $G[V(G)\setminus S]$ is an extremal $L$-free graph, and
\item [(ii)] $G$ contains all possible edges between $S$ and $V(G)\setminus S$.
\end{itemize}	
Because any copy of $L$ in $G$ must contains at least one vertex in $S$, we see that $G$ contains at most $k-1$ vertex-disjoint copies of $L$.
	Moreover, we have
	$$e(G)=\binom{cn-1}{2}+(cn-1)(n-cn+1)+ex(n-cn+1,L)=(1-\frac{(1-c)^{2}}{\chi(L)-1}+o(1))\binom{n}{2}.$$
	Since the family consisting of all subgraphs of $G$ provides a lower bound to the corresponding dual concept $D_{k\cdot L}(n)$, by Lemma \ref{Dual.L}, we have
	$$M_{k\cdot L}(n)\leq 2^{\binom{n}{2}-e(G)}=2^{(\frac{(1-c)^{2}}{\chi(L)-1}+o(1))\binom{n}{2}}.$$
	Therefore, it follows that for $k=cn$,
	$$\lim_{n\to\infty}\frac{\log M_{k\cdot L}(n)}{\binom{n}{2}}\leq \frac{(1-c)^{2}}{\chi(L)-1},$$
completing the proof.
\end{proof}

\subsection{Proof of Theorem \ref{ComBi.Th}}
	
\begin{proof}[Proof of Theorem \ref{ComBi.Th}]
	First we prove the case $t=o(\log n)$.
	By Theorem \ref{KST.Th}, we have
	$$ex(n,K_{t,t})\leq\frac{1}{2}(t-1)^{\frac{1}{t}}n^{2-\frac{1}{t}}+\frac{1}{2}(t-1)n.$$
	Let $t=\varepsilon\log n$, then we have
	$$ex(n,K_{t,t})\leq2^{-1-\frac{1}{\varepsilon}}(\varepsilon\log n-1)^{\frac{1}{\varepsilon}\log_{n}2}n^{2}+o(n^{2}).$$
	Let $\varepsilon\to 0$, we have $ex(n,K_{t,t})=o(n^{2})$.
	Therefore, the number of $K_{t,t}$-free graph on $n$ labeled vertices is at most $\binom{\binom{n}{2}}{o(n^{2})}=2^{o(1)\binom{n}{2}}$.
	Moreover, $F_{n}(K_{t,t})\geq 2^{ex(n,K_{t,t})}$. So we have $F_{n}(K_{t,t})=2^{o(1)\binom{n}{2}}$.
	
	Let $G_{K_{t,t}}$ denote the graph whose vertices are all possible graphs on $n$ labeled vertices and two vertices are connected if and only if their symmetric difference contains a copy of $K_{t,t}$. Then $M_{K_{t,t}}(n)=\alpha(G_{K_{t,t}})$ and $G_{K_{t,t}}$ is $F_{n}(K_{t,t})$-regular. So we have
	 $$M_{K_{t,t}}(n)=\alpha(G_{K_{t,t}})\geq\frac{v(G_{K_{t,t}})}{\Delta(G_{K_{t,t}})+1}=\frac{2^{\binom{n}{2}}}{F_{n}(K_{t,t})+1}=2^{(1+o(1))\binom{n}{2}}.$$
	 Since $M_{K_{t,t}}(n)\leq 2^{\binom{n}{2}}$, for $t=o(\log n)$ we have
	 $$\lim_{n\to\infty}\frac{\log M_{K_{t,t}}(n)}{\binom{n}{2}}=1.$$

It remains to consider the case when $t=c\log n$ for some constant $c>0$.
In this case, we only need to construct an $n$-vertex $K_{t,t}$-free graph $G$ with at least $2^{-\frac{2}{c}}\binom{n}{2}$ edges.
	 If such graph $G$ exists, then the family consisting of all subgraphs of $G$ provides a lower bound to the corresponding dual concept $D_{K_{t,t}}(n)$. So by Lemma \ref{Dual.L}, we have $M_{K_{t,t}}(n)\leq2^{\binom{n}{2}-e(G)}$. Therefore, we have
	 $$\lim_{n\to\infty}\frac{\log M_{K_{t,t}}(n)}{\binom{n}{2}}\leq1-2^{-\frac{2}{c}}$$
	 for $k=c\log n$.
	
	 Now we use probabilistic methods to construct such a graph $G$.
	 Let $\delta=2^{-\frac{2}{c}}$ and consider the Erd\"{o}s-R\'{e}nyi random graph $G(n,\delta)$ (i.e. an n-vertex graph in which each possible edge is present independently with probability $\delta$). Let $X$ be the number of $K_{t,t}$ in $G(n,\delta)$, we have
	 $$\mathbb{E}[X]=\frac{1}{2}\binom{n}{2t}\binom{2t}{t}\delta^{t^{2}}<n^{2t}\delta^{t^{2}}=(n^{2}\delta^{t})^{t}.$$
	 Since $\delta^{t}=2^{-\frac{2}{c}c\log n}=n^{-2}$, we have $\mathbb{E}[X]<1$.
	 By average, there exists a graph $G'$ such that $e(G')-X\geq \mathbb{E}[e(G(n,\delta))-X]>\delta\binom{n}{2}-1$.
	 Let $G$ be obtained from $G'$ by deleting one edge for each copy of $K_{t,t}$ in $G'$, then $G$ is an $n$-vertex $K_{t,t}$-free graph with at least $2^{-\frac{2}{c}}\binom{n}{2}$ edges. We have completed the proof.
\end{proof}

\subsection{Proof of Theorem \ref{AnoBi.Th}}	
	
\begin{proof}[Proof of Theorem \ref{AnoBi.Th}]

	Firstly, we claim that $ex(n,L(r,m))=o(n^{2})$.
	
	Let $\alpha=n^{-\frac{\varepsilon^{2}}{2r}}$, $t=\frac{r}{\varepsilon}$ and $u=m$. Then for sufficiently large $n$, we have $\alpha^{t}n-\binom{n}{r}(\frac{m}{n})^{t}\geq u$. So, for sufficiently large $n$ and any $n$-vertex graph $G$ with at least $\frac{1}{2}n^{2-\frac{\varepsilon^{2}}{r}}$ edges, there exists $U\subseteq V(G)$ with $|U|\geq u$ such that any $r$-set $S\subseteq U$ has at least $m$ common neighbors.
	
	Now we are going to find an $L(r,m)=(A\cup B,E)$ in such $G$.
	Let $\phi$ be any injection from $B$ to $U$, we only need to extend it to an injection from $A\cup B$ to $V(G)$ such that for any edge $ab$ in $L(r,m)$, $\phi(a)\phi(b)$ is an edge in $G$. Let $A'$ be a subset of $A$ and assume that we have already extend $\phi$ to an injection from $A'\cup B$ to $V(G)$ such that for any edge $ab$ between $A'$ and $B$, $\phi(a)\phi(b)$ is an edge in $G$. Take an vertex $v\in A\setminus A'$, then $\phi(N_{L(r,m)}(v))$ is a subset of $U$ with cardinality at most $r$. Take an $r$-set $S\subseteq U$ with $S'\subseteq S$ and let $T$ denote the set of common neighbors of $S$ in $G$. Then $|T|\geq m=|V(L(r,m))|$. Therefore $T\setminus \phi(A'\cup B)$ is not empty. We can choose an vertex $x$ in $T\setminus \phi(A'\cup B)$ and let $\phi(v)=x$. Then we can check that $\phi$ is an injection from $A'\cup\{v\}\cup B$ to $V(G)$ with the property that for any edge $ab$ between $A'\cup\{v\}$ and $B$, $\phi(a)\phi(b)$ is an edge in $G$. By induction, we get a desired $\phi$.
	
	Since for sufficiently large $n$, any $n$-vertex graph $G$ with at least $\frac{1}{2}n^{2-\frac{\varepsilon^{2}}{r}}$ edges contains a copy of $L(r,m)$, we know that $ex(n,L(r,m))=O(n^{2-\frac{\varepsilon^{2}}{r}})=o(n^{2})$. Then by the same argument in Theorem \ref{ComBi.Th}, we have
$\lim_{n\to\infty}\frac{\log M_{L(r,m)}(n)}{\binom{n}{2}}=1.$
\end{proof}

\subsection{Proof of Proposition \ref{grid.Pro}}

\begin{proof}[Proof of Proposition \ref{grid.Pro}]
	For any integers $m,n\geq 3$, $G_{m,n}$ is a 4-regular graph. Let $\mathcal{G}=\{G_{1},G_{2},\dots,G_{s}\}$ be a family of spanning subgraphs of $G_{m,n}$ such that the symmetric difference of any two members in $\mathcal{G}$ is connected. For each $1\leq i\leq s$, let $N_{i}$ denote the set of the neighbors of $(1,1)$ in $G_{i}$. Then $N_{1},N_{2},\dots,N_{s}$ must be pairwise distinct. If not, we may assume that $N_{1}=N_{2}$, then $(1,1)$ will be an isolated vertex in the symmetric difference of $G_{1}$ and $G_{2}$, which contradicts to the definition of $\mathcal{G}$. Note that $N_{1},N_{2},\dots,N_{s}$ are all subsets of $\{(1,2),(1,n),(2,1),(m,1)\}$. So we have $s\leq2^{4}=16$. Therefore, $M_{\mathcal{F}_{c}}(G_{m,n})\leq 16$.

The tight construction for the case $m=n=3$ can be found in the Appendix.
\end{proof}

\section{Concluding Remarks}

\begin{itemize}

	\item
	Theorem \ref{Tree.Th} shows that for all $3\leq\ell \leq\frac{n-1}{12\log n}+2$, the family $\mathcal{T}_{\ell}$ could not provide a positive answer to Problem~\ref{P3}. 
Actually, this is also true for all $\frac{n-1}{12\log n}+2\leq\ell\leq\frac{n+9}{6}$. To be more precise, if $n=2^{k}-1$ for some positive integer $k$ and Conjecture \ref{P1FC} holds for $n+1$, then for any integer $\ell\in[3,\frac{n+9}{6}]$, we have $M_{\mathcal{F}_{\mathcal{T}_{\ell}}}(n)\geq 2^{n-k-1}$. The proof of this statement is nearly the same as the proof of Theorem \ref{Tree.Th}. Firstly, we partition the edge set of $K_{n+1}$ into $n$ perfect matchings $M_{1},M_{2},\dots,M_{n}$ such that the union of any two of them forms a Hamiltonian cycle in $K_{n+1}$. 
Then we delete the edge adjacent to $n+1$ in $M_{i}$ for each $1\leq i\leq n$. 
Secondly, we take a subfamily $\mathcal{G}$ of the power set $2^{\{M_{1},\dots,M_{n}\}}$ with cardinality $2^{n-k-1}$ such that the symmetric difference of any two members in $\mathcal{G}$ contains at least 3 matchings in $\{M_{1},\dots,M_{n}\}$. 
This is guaranteed by the existence of Hamming code of length $2^{k}-1$ for $k\geq2$ (for a nice quick account on Hamming codes see e.g. \cite{Ber}). Finally, we only need to find a spanning tree with exactly $\ell$ leaves in the union of 3 matchings,
two of which forms a Hamiltonian path in $K_n$ and the other one is viewed as a set of $\frac{n-1}{2}$ additional edges.

\item We prove the phase transitions in Theorems \ref{kcopy.Th} and \ref{kvexd.Th}.
However, it seems that to find the precise rate limit for all values of $k$ is a difficult task.
To do so, one needs to obtain further knowledge about the extremal structure of $n$-vertex graphs that contain at most $cn^{v(L)}$ copies (or $cn$ vertex-disjoint copies) of $L$. 
However, this remains an open problem in extremal graph theory and only very few cases of determining the extremal structures are currently known (e.g.\cite{Raz},\cite{ABHP} and \cite{HLLYZ}).
	
	\item
	The construction in Proposition \ref{grid.Pro} shows that the general upper bound $M_{\mathcal{F}_{c}}(G_{m,n})\leq 16$ is sharp for $m=n=3$. It would be of interest to investigate whether this can occur for other values of $m$ and $n$. We leave this as an open problem.
	\begin{prob}
		Is it true that $M_{\mathcal{F}_{c}}(G_{m,n})=16$ holds for all $m,n\geq 3$?
	\end{prob}
	
\end{itemize}

\section*{Appendix}
\begin{figure}[H]
\centering

\subfigure[$G_{1}$]{
	\centering
	\includegraphics[scale=0.45]{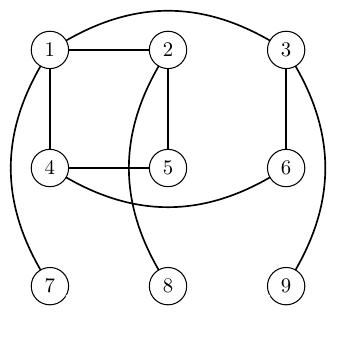}
}
\subfigure[$G_{2}$]{
	\centering
	\includegraphics[scale=0.45]{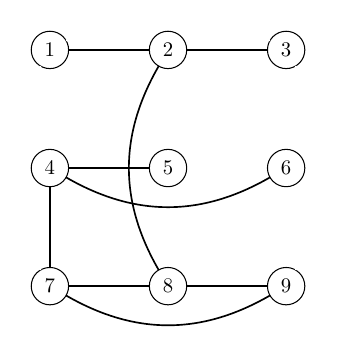}
}
\subfigure[$G_{3}$]{
	\centering
	\includegraphics[scale=0.45]{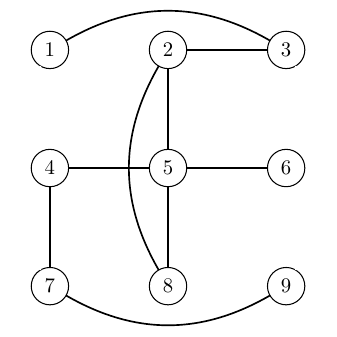}
}
\subfigure[$G_{4}$]{
	\centering
	\includegraphics[scale=0.45]{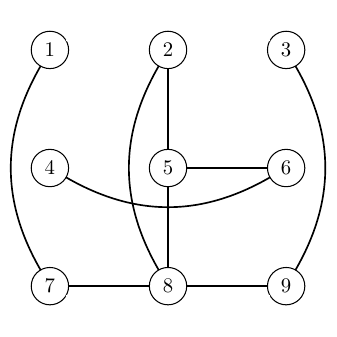}
}
\centering
\subfigure[$G_{5}$]{
	\centering
	\includegraphics[scale=0.45]{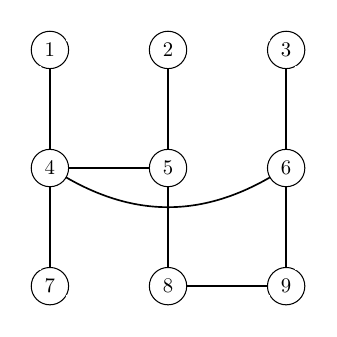}
}

\subfigure[$G_{6}$]{
	\centering
	\includegraphics[scale=0.45]{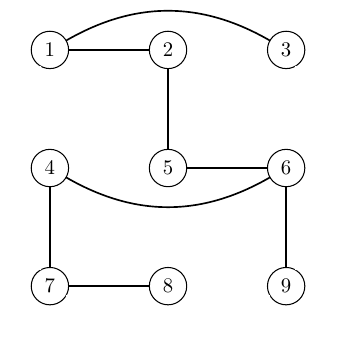}
}
\subfigure[$G_{7}$]{
	\centering
	\includegraphics[scale=0.45]{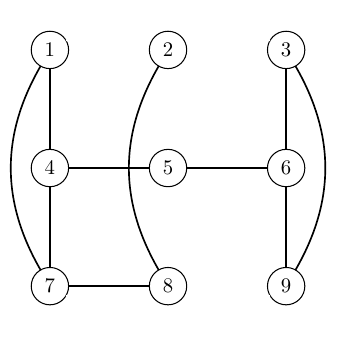}
}
\subfigure[$G_{8}$]{
	\centering
	\includegraphics[scale=0.45]{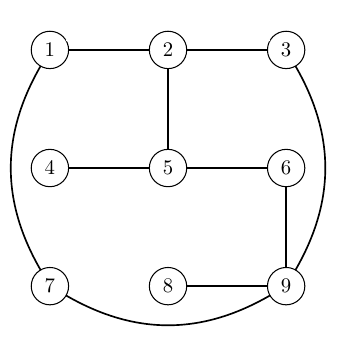}
}
\subfigure[$G_{9}$]{
	\centering
	\includegraphics[scale=0.45]{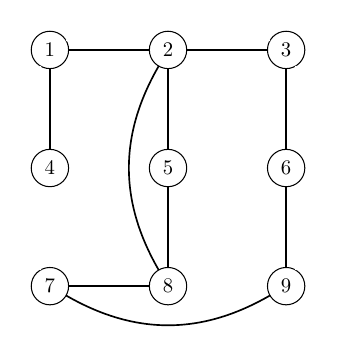}
}
\subfigure[$G_{10}$]{
	\centering
	\includegraphics[scale=0.45]{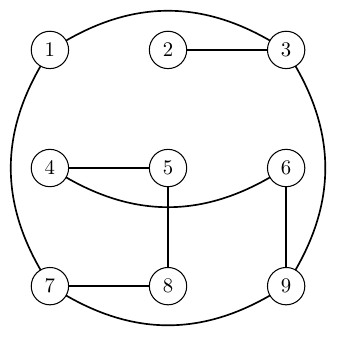}
}

\subfigure[$G_{11}$]{
	\centering
	\includegraphics[scale=0.45]{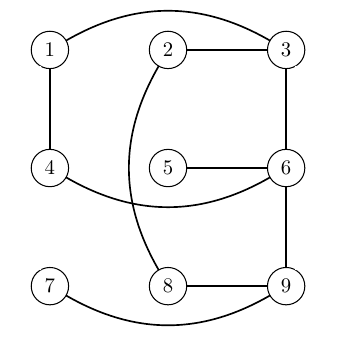}
}
\subfigure[$G_{12}$]{
	\centering
	\includegraphics[scale=0.45]{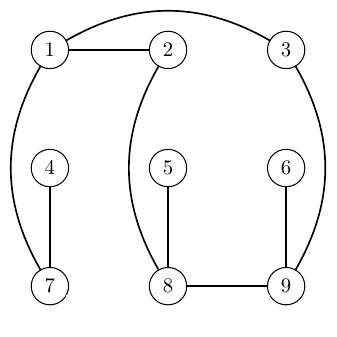}
}
\subfigure[$G_{13}$]{
	\centering
	\includegraphics[scale=0.45]{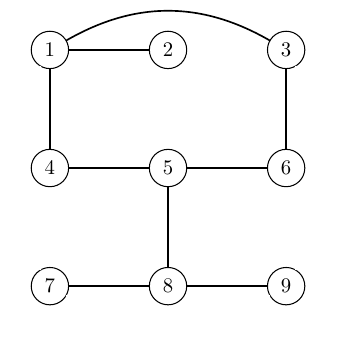}
}
\subfigure[$G_{14}$]{
	\centering
	\includegraphics[scale=0.45]{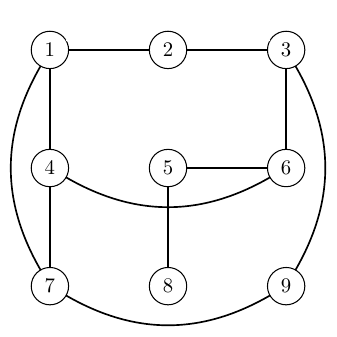}
}
\subfigure[$G_{15}$]{
	\centering
	\includegraphics[scale=0.45]{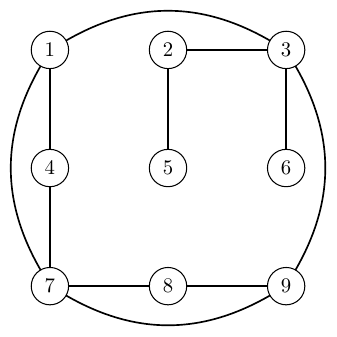}
}

\subfigure[$G_{16}$]{
	\centering
	\includegraphics[scale=0.45]{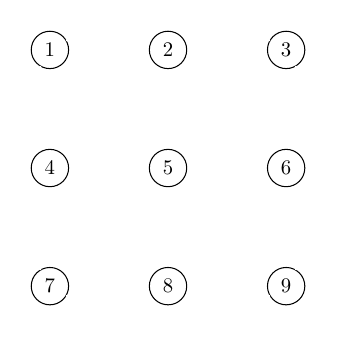}
}
\end{figure}	

\end{document}